\documentclass[a4paper,11pt]{amsart}
\usepackage{amsfonts}
\usepackage{amsthm}
\usepackage{amsmath}
\usepackage{amssymb}
\usepackage{graphicx}
\usepackage{enumitem}
\usepackage{color}

\usepackage{stmaryrd}

\definecolor{sealbrown}{rgb}{0.2, 0.08, 0.08}
\definecolor{OwlGreen}{RGB}{90,168,0}
\definecolor{OwlRed}{RGB}{255,92,168}
\usepackage[colorlinks=true]{hyperref}
\hypersetup{citecolor=OwlGreen, linkcolor=blue}

\newtheorem{theorem}{Theorem}[section]
\newtheorem{corollary}[theorem]{Corollary}
\newtheorem{lemma}[theorem]{Lemma}
\newtheorem{proposition}[theorem]{Proposition}
\theoremstyle{definition}
\newtheorem{definition}[theorem]{Definition}

\newtheorem{remark}[theorem]{Remark}

\numberwithin{equation}{section}

\newcommand{\co}{\mathrm{c}_0}

\newcommand{\vertiii}[1]{{\left\vert\kern-0.35ex\left\vert\kern-0.35ex\left\vert #1 \right\vert\kern-0.35ex\right\vert\kern-0.35ex\right\vert}}

\title[Retraction methods and fixed point property]{Retraction methods and fixed point free maps with null minimal displacements on unit balls}

\author[C. S. Barroso]{Cleon S. Barroso}

\address[C. S. Barroso]{Department of Mathematics, Federal University of Cear\'a, 60455-360, Campus do Pici, Av. Humberto Monte S/N, Fortaleza, Brazil}
\email{{\tt cleonbar@mat.ufc.br}}

\author[V. Ferreira]{Valdir Ferreira}

\address[V. Ferreira]{Centro de Ci\^encias e Tecnologia, Universidade Federal do Cariri, Cidade Universit\'aria s/n, 63048-0808, Juazeiro do Norte, CE, Brazil}
\email{\tt valdir.ferreira@ufca.edu.br}



\begin{document}

\begin{abstract}
In this paper we consider the class of Lipschitz maps on the unit ball $B_X$ of a Banach space $X$, and the question we deal with is whether for any $\lambda>1$ there exists a $\lambda$-Lipschitz fixed-point free mapping $T\colon B_X\to B_X$ with $\mathrm{d}(T,B_X)=0$. We also consider its H\"older version. New related results are obtained. We show that if $X$ has a spreading Schauder basis then such mappings can always be built, answering a question posed by the first author in \cite{Bar}. In the general case, using a recent approach of R. Medina \cite{M} concerning H\"older retractions of $(r_n)$-flat closed convex sets, we show that for any decreasing null sequence $(r_n)\subset \mathbb{R}$ and $\alpha\in (0,1)$, there exists a fixed-point free mapping $T$ on $B_X$ so that $\|T^nx - T^n y\|\leq r_n(\| x - y\|^\alpha +1)$ for all $x, y\in B_X$ and $n\in\mathbb{N}$. 
\end{abstract}

\maketitle


\section{Introduction}
Let $X$ be a real Banach space and $B_X$ denote its closed unit ball. For a convex subset $C$ of $X$, denote by $\mathcal{B}(C)$ the family of all bounded, closed convex subsets  of $C$. In \cite{LS} P. K. Lin and Y. Sternfeld proved that for any noncompact (in norm) set $K\in \mathcal{B}(X)$ there exists a Lipschitz map $T\colon K\to K$ with positive {\it minimal displacement}. That is, $\mathrm{d}(T, K)=\inf_{x\in K}\| x - T(x)\|>0$. This shows in particular that $\digamma(T)=\emptyset$, where $\digamma(T)$ denotes the fixed-point set of $T$. For any such $K$, two questions that naturally arise from the context in \cite{LS} are:
\vskip .15cm
\hskip .15cm {\bf ($\mathcal{Q}1$)} {\it Is there a Lipschitz map $T\colon K\to K$ with $\mathrm{d}(T, K)=0$ and $\digamma(T)=\emptyset$?}
\vskip .2cm
\hskip .15cm {\bf ($\mathcal{Q}2$)} {\it Is there a {\it uniformly} Lipschitz map $T\colon K\to K$ with $\digamma(T)=\emptyset$?}

\smallskip  

\noindent It is worth highlighting that these issues have been studied directly or indirectly in several works, which in a way illustrate how challenging they are. In retrospect, the class of maps satisfying $\mathrm{d}(T, K)=0$ includes three important subclasses, each of which playing a distinct role in metric fixed point theory. Namely, affine maps, nonexpansive (i.e. $1$-Lipschitz) mappings and those that are asymptotically regular (i.e. $\| T^nx - T^{n+1} x\| \to 0$, $x\in K$). The literature exploring different aspects of the {\it fixed point property} (FPP) and the behavior of $\mathrm{d}(T, K)$ is quite vast, cf. e.g. \cite{BF, BenJap, GoKi, GMMV, Ki3, Pia} and references therein. In \cite{BenJap} it was noted that the Lipschitz constant in Lin-Sternfeld's result can be made to be as close to $1$ as desired. The argument uses convex combinations of the form $(1- \lambda)T + \lambda I$ with $\lambda\approx 1$. As far as we known, however, questions ($\mathcal{Q}1$-$\mathcal{Q}2$) remain still open. It is not clear, e.g., whether Lin-Sternfeld's approach can be refined to solve them. Moreover, despite the growing interest in these questions, there seem to be no optimal answers even when $K=B_X$. In contrast, it is known, for example, that $B_{\ell_\infty}$ and $B_{\ell_1}$ both have the FPP for nonexpansive maps, whereas $B_{\co}$ does not \cite{GoKi}. Also, Lin \cite{Lin2} displayed an $\ell_1$ renorming $\|\cdot\|_\gamma$ ensuring in particular that $B_{(\ell_1,\| \|_\gamma)}$ has the FPP for nonexpansive maps. 

\smallskip 
  
\noindent In this paper we establish some new contributions to these problems and explore some of their consequences. Here we will focus on the existence of a set $K$ and a mapping $T$ that fulfill the properties prescribed in ($\mathcal{Q}1$) or ($\mathcal{Q}2$). In this regard, the approach for $K=B_X$ includes using retraction methods and the consequences concern the H\"older version of Lin-Sternfeld's result, cf. \cite{Bar}: If $X$ is infinite dimensional, $\alpha\in (0,1)$ and $\lambda>0$, then there is a mapping $T\colon B_X\to B_X$ with $\digamma(T)=\emptyset$, $\mathrm{d}(T, B_X)>0$ and $\|T(x) - T(y)\|\leq \lambda \|x - y\|^\alpha$ for all $x, y\in B_X$. Recall \cite{Bar} that, for $\alpha\in (0,1]$ and $\lambda>0$, a mapping $T\colon K\to K$ is called:
\begin{itemize}\setlength\itemsep{1mm}
\item $\alpha$-H\"older nonexpansive if $\|T(x) - T(y)\|\leq \|x - y\|^\alpha$ for all $x, y\in K$.
\item $\alpha$-H\"older $\lambda$-Lipschitz  if  $\|T(x) - T(y)\|\leq \lambda \|x - y\|^\alpha$ for all $x, y\in K$.
\item Uniformly $\alpha$-H\"older $\lambda$-Lipschitz if $\| T^n(x) - T^n(y)\|\leq \lambda \|x - y\|^\alpha$ for all $x, y\in K$ and $n\in\mathbb{N}$.
\end{itemize} 

\smallskip 

\noindent Fixed point issues for a larger class of H\"older maps were first addressed by Kirk \cite{Ki3}. His results only yield growth estimates for $\mathrm{d}(T,K)$ (cf. \cite[Proposition 2.2, Theorems 2.3, 2.4, 4.1, 4.4 and 4.5]{Ki3}). However, as the class of H\"older maps considered here is slightly smaller, it is expected that ($\mathcal{Q}1$--$\mathcal{Q}2$) can also be solved in their H\"older's versions. Specifically, we shall try to use H\"older retractions to provide answers for the following issues: under what conditions does there exist for any $\alpha\in (0,1$) and $\lambda>0$, a fixed-point free mapping $T\colon B_X\to B_X$ with $\mathrm{d}(T,B_X)=0$ and with the property of being: 

\smallskip 

\hskip .2cm {\bf ($\mathcal{Q}3$)} {\it $\alpha$-H\"older $\lambda$-Lipschitz?}

\vskip .2cm 

\hskip .2cm {\bf ($\mathcal{Q}4$)} {\it uniformly $\alpha$-H\"older $\lambda$-Lipschitz?}

\smallskip 

\noindent Mostly influenced by Kirk's work, several partial answers to these questions have been provided in \cite{Bar}. Here we also obtain new related results. 

\smallskip 

\subsection{Organization} The plan of this paper is the following. In Section \ref{sec:2} we set up notation and some basic terminologies. Our results are delivered in Sections \ref{sec:3} and \ref{sec:4}. The paper ends in Section \ref{sec:5} with some concluding remarks.

\smallskip 
 
\subsection*{Acknowledgments} A significant part of this work was presented at the Brazilian Workshop in Banach Spaces, Butant\~a Edition, USP - S\~ao Paulo, December 05-10, 2022. The first author is grateful to C. Brech and V. Ferenczi by the opportunity and support. The authors wish to thank R. Medina for helpful comments. The first version of this work was done when the second author was visiting the Department of Mathematics of the UFC, during July 10--15, 2023. He thanks Professor Ernani Ribeiro for the support and attention. The authors are grateful to the anonymous referee for the careful reading and helpful comments that helped us not only to improve the presentation, but also to correct some inaccuracies in the previous version of this manuscript.
 
\smallskip 

\section{Preliminaries}\label{sec:2}

The terminology used here follows \cite{AK, FHHMZ, LTI, M}. Throughout this paper all Banach spaces are infinite dimensional and real. For $r>0$, $B_X(r)$ denotes the closed ball in $X$ with center $0$ and radius $r$. $\ell_\infty$, $\mathrm{c}$ and $\mathrm{c}_0$ denote the classical sequence spaces of all bounded, convergent and null sequences, respectively. For $1\leq p<\infty$, $\ell_p$ denotes the space of all $p$-absolutely convergent series. By $\mathrm{c}_{00}$ we denote the space of eventually null sequences. A basic sequence $(x_n)_{n=1}^\infty$ in $X$ is called {\it semi-normalized} if there exist $A, B>0$ such that $A\leq \| x_n\|\leq B$ for all $n$. Here $\llbracket x_n\rrbracket$ denotes the closed linear span of $(x_n)_{n=1}^\infty$. If $(x_n)_{n=1}^\infty$ and $(y_n)_{n=1}^\infty$ are basic sequences in Banach spaces $X$ and $Y$, respectively, we say that $(x_n)_{n=1}^\infty$ $D$-dominates $(y_n)_{n=1}^\infty$ for $D\geq 1$, denoted by $(y_n)_{n=1}^\infty \lesssim_D (x_n)_{n=1}^\infty$, if the linear mapping $\mathcal{L}\colon \llbracket x_n\rrbracket\to \llbracket y_n \rrbracket$ given by $\mathcal{L}(x_n)= y_n$ for all $n$, is bounded with $\|\mathcal{L}\|\leq D$. $(x_n)_{n=1}^\infty$ is said to be $(A,B)$-{\it equivalent} to $(y_n)_{n=1}^\infty$, for $A, B>0$, if $(y_n)_{n=1}^\infty \lesssim_{A} (x_n)_{n=1}^\infty\lesssim_B (y_n)_{n=1}^\infty$. The {\it fundamental function} of a basic sequence $(x_n)_{n=1}^\infty$ is the function $\Phi\colon \mathbb{N}\to \mathbb{R}^+$ given by $\Phi(n)= \|\sum_{i=1}^n x_i \|$. A basis is called {\it subsymmetric} (resp. $1$-subsymmetric) if it is unconditional and equivalent (resp. $1$-unconditional and $1$-equivalent) to each of its subsequences. According to \cite[Definition 2.1]{Anso} the basis $(x_n)_{n=1}^\infty$ is said to be {\it lower subsymmetric} if it is unconditional and dominates all of its subsequences; and it is called {\it $C$-lower subsymmetric} if for every sequence of signs $(\epsilon_n)$ and every increasing map $\phi\colon \mathbb{N}\to \mathbb{N}$, $(\epsilon_n x_{\phi(n)})_{n=1}^\infty \lesssim_C (x_n)_{n=1}^\infty$. We will refer to a basic sequence $(x_n)_{n=1}^\infty$ as being {\it unconditional shift} (resp. {\it $1$-unconditional shift\,}) if it is unconditional (resp. $1$-unconditional) and dominates its right-shift sequence $(x_{n+1})_{n=1}^\infty$.


\medskip 

\section{On H\"older retractions of $(r_n)$-flat convex sets}\label{sec:3}

The goal of this section is to review the retraction approach \cite{M}. What really motivates us to do this is the bet on possible connections with the metric fixed point theory. Let us begin with some notation. Let $(M, d_M)$ and $(N, d_N)$ be two metric spaces and $f\colon M\to N$ be an arbitrary mapping. The modulus of continuity of $f$ is the function $\omega_f\colon [0,\infty) \to [0,\infty)$ given by
\[
\omega_f(t) = \sup\big\{ d_N\big( f(x), f(y)\big) \, \colon \, d_M(x, y)\leq t\big\}, 
\]
where the supremum is taken to be infinite whenever it does not exist. Note that 
\[
d_N(f(x), f(y))\leq \omega_f( d_M(x,y))\quad\text{for all } x, y\in M.
\]
The mapping $f$ is said to be:
\begin{itemize}\setlength\itemsep{1.5mm}
\item uniformly continuous if $\omega_f$ is continuous at $t=0$.
\item $\omega$-Lipschitz if $\omega_f(t)< \infty$ for all $t\geq 0$.
\item uniformly $\omega$-Lipschitz if $\sup_{n\in\mathbb{N}}\omega_{f^n}(t)<\infty$ for all $t\geq 0$, where $f^n$ is the $n$-\textrm{th} iterate of $f$.
\item $\omega$-H\"older Lipschitz if for some $\alpha\in (0,1)$ and some constant $C>0$ (H\"older-Lipschitz constant of $f$) one has $\omega_f(t)\leq C t^\alpha$ for all $t\geq 0$.  
\end{itemize} 
 
\noindent A retraction from a metric space $(M, d)$ onto a subset $N\subset M$ is a mapping $R\colon M\to N$ satisfying $R(x) =x$ for every $x\in N$. The image of a retraction is called a retract. Further, for $\alpha\in (0,1)$, $N$ is said to be an absolute $\alpha$-H\"older retract if it is a $\alpha$-H\"older retract of every metric space containing it. 

\vskip .1cm 

\noindent Recall \cite[p. 3]{M} that for $a, b\in \mathbb{R}^+$, a subset $N\subset M$ is called an $(a, b)$-net of $M$ if the following properties hold:
\begin{itemize}\setlength\itemsep{1.5mm}
\item $N$ is $a$-separated, that is, $d(x, y)\geq a$ for every $x\neq y\in N$.
\item $N$ is $b$-dense, that is, for every $x\in M$ there is $y\in N$ such that $d(x, y) \leq b$. 
\end{itemize}

\vskip .1cm 

\noindent Let now $(X, \|\cdot\|)$ be an infinite dimensional normed space and $E=(E_n)_{n=1}^\infty$ a sequence of $n$-dimensional subspaces of $X$. 

\begin{definition} Given a nonempty subset $K$ of $X$, the heights $(h^E_n)_{n\in\mathbb{N}}$ of $K$ relative to $E=(E_n)_{n=1}^\infty$ (\cite[Definition 2.2]{M}) are defined as
\[
h^E_n:=\sup\{ \mathrm{dist}(x, K\cap E_n) \, \colon \, x\in K\}.
\]
\end{definition}

\noindent The following notion of $(r_n)$-flatness is slightly weaker than the one given in \cite[Definition 2.2]{M}. 

\begin{definition}\label{dfn:1sec3} A nonempty set $K\subset X$ is called $(r_n)$-flat for some null sequence of positive numbers $(r_n)_{n=1}^\infty \subset \mathbb{R}^+$ if there is a sequence $E=(E_n)_{n=1}^\infty$ of $n$-dimensional subspaces of $X$ so that the following conditions hold true for every $n\in \mathbb{N}$:
\begin{itemize}\setlength\itemsep{1.5mm}
\item[(i)] $K\cap E_n$ is nonempty and compact.
\item[(ii)] $h^E_n \leq r_n$.
\end{itemize}
\end{definition}

\noindent The main result of this section (compare with \cite[Theorem 2.9]{M}) reads.

\begin{theorem}[R. Medina]\label{thm:1sec3} Let $\alpha\in (0,1)$. Then every $(20^{\frac{n}{\alpha-1}})$-flat closed convex subset $K$ of an infinite dimensional normed space $X$ is a $\alpha$-H\"older retract of $X$ with H\"older-Lipschitz constant $1520\times 20^{2-\alpha}$.
\end{theorem}

\noindent It is worthy to note that this statement differs from the one in \cite{M} in that $X$ need not be complete, nor does $K$ need to be compact. However the proof is virtually the same (cf. \cite[Lemmas 2.4, 2.5, Proposition 2.6 and Theorems 2.7 and 2.9]{M}). So, we shall only outline its main points. Let $K$ be an arbitrary $(r_n)$-flat closed convex subset of $X$ where $(r_n)_{n=1}^\infty\subset\mathbb{R}^+$ is decreasing and null. Take $E=(E_n)_{n=1}^\infty$ to be as in Definition \ref{dfn:1sec3}. For $\varepsilon>0$ define $n(\varepsilon):=\min\{ n\in\mathbb{N}\cup \{0\} \,\colon\, r_n\leq \varepsilon\}$. Since each set $K\cap E_{n(\varepsilon)}$ is nonempty and compact, we can consider a $(\varepsilon, \varepsilon)$-net $N_\varepsilon = (x_i^\varepsilon)_{i=1}^{m_\varepsilon}$ of $K\cap E_{n(\varepsilon)}$. Using (ii)-Definition \ref{dfn:1sec3} one can easily verify that $N_\varepsilon$ is also a $(\varepsilon, 2\varepsilon)$-net of $K$ for every $\varepsilon>0$. Now for every $n\in\mathbb{Z}$ set $\varepsilon_n=2^{-n}$, $N_n:=N_{\varepsilon_n}$ and $x^n_i:=x_i^{\varepsilon_n}$ for all $i\in\{ 1, \dots, m_{\varepsilon_n}:=m_n\}$, where $N_{\varepsilon_n}=(x_i^{\varepsilon_n})_{i=1}^{m_{\varepsilon_n}}$. Also consider the sets
\[
\tilde{V}_i^n=\big\{ x\in X\colon \varepsilon_n \leq d(x, K)< \varepsilon_{n-1}, \, d(x,x^n_i)=\min_j d(x, x^n_j)\big\}
\]
and
\[
V^n_i=\{ x\in X\,\colon\, d(x, \tilde{V}^n_i)\leq \varepsilon_{n+1}\},
\]
where $d(x, A)=\mathrm{dist}(x, A)$ whenever $A\subset X$ is a nonempty and $d(x, y)=\|x - y\|$ for all $x, y\in X$. As $K$ is closed {\it in} $X$, $d(x, K)>0$ for all $x\in K^c:=X\setminus K$. 

\vskip .2cm 

\noindent The proofs of the next results are the same as that given in \cite[Lemmas 2.4 and 2.5, Fact 2.1 and Proposition 2.6]{M}. As for Lemma 2.4, it is worth highlighting the importance of space $E_n$ being $n$-dimensional. This is somewhat technical and is used in equation (2.5) of \cite{M} when proving the crucial item (2) (cf. \cite[p.6]{M}). 


\begin{lemma}\label{lem:1sec3} Let $X$ and $K$ be as above. Then $\mathcal{F}=\{ V^n_i\,\colon \, n\in\mathbb{Z},\, i\in \{1, \dots, m_n\}\}$ defines a locally finite cover of $X\setminus K$. Moreover, if $x\in V^n_i \in \mathcal{F}$ and $x\in X\setminus K$ then
\begin{itemize}\setlength\itemsep{1.5mm}
\item[\it{(1)}] $d(x, K)/5\leq \| x - x^n_i\|\leq 9d(x, K)$.
\item[\it{(2)}] $\#\{ V\in \mathcal{F}\,\colon\,x \in V\}\leq 5\times 20^{n(d(x, K)/10)}$.
\item[\it{(3)}] $d(x, K)/4\leq \max_{V\in\mathcal{F}} d(x, V^c)\leq d(x, K)$. 
\item[\it{(4)}] For $n\in\mathbb{Z}$, $i\in\{ 1, \dots, m_n\}$ and $x\in K^c$, set
\[
\varphi^n_i(x) = \frac{d(x, (V^n_i)^c)}{\sum_{k, j} d(x, (V^k_j)^c)}.
\]
Then $\{ \mathcal{F},( \varphi^n_i)_{i,n}\}$ defines a partition of unity for $X\setminus K$. In addition,
\[
\sum_{i,n}|\varphi^n_i(x) - \varphi^n_i(y)|\leq \frac{40\times (20^{n(d(x,K)/10)} + 20^{n(d(y,K)/10)})}{\max\{d(x,K), d(y, K)\}}\|x - y\|,\quad x, y\in K^c.
\]
\end{itemize}
\end{lemma}

\vskip .2cm

\noindent Now consider the map $R\colon X\to K$ given by 
\[
R(x)=\left\{
\begin{aligned}
&\sum_{i, n} \varphi^n_i(x) x^n_i & \text{ if } & x\in X\setminus K,\\
&\,x & \text{ if } & x\in K.
\end{aligned}
\right.
\]

\vskip .2cm

\begin{proposition}\label{prop:1sec3} Let $X$, $K$ and $R$ be as above. Then the following holds true:
\begin{itemize}\setlength\itemsep{1.5mm}
\item[(i)] $\| R(x) - x\| \leq 9 d(x, K)$ for all $x\in X$.
\item[(ii)] For every $x, y\in X\setminus K$, 
\[
\| R(x) - R(y)\| \leq 760\big( 20^{n(d(x, K)/10)} + 20^{n(d(y, K)/10)}\big) \| x - y\|.
\]
\end{itemize}
\end{proposition}

\vskip .1cm
\noindent As pointed out in \cite{M}, even though Lemma \ref{lem:1sec3} describes the Lipschitz behaviour of  $\{\mathcal{F}, (\varphi^n_i)_{i,n}\}$, Proposition \ref{prop:1sec3} shows how the retraction $R$ loses its Lipschitzness when the points get close to $K$. However, the constraint imposed by the sequence $(r_n)_{n=1}^\infty$ over the heights $(h^E_n)_{n=1}^\infty$ is important because it provides useful property for the growth of the modulus of continuity of $R$. This is witnessed in the next result (cf. details in \cite[Theorem 2.7]{M}):

\vskip .2cm 

\begin{theorem}\label{thm:2sec3} Let $X$, $K$ and $R$ be as above. Then for every $t\in\mathbb{R}^+$,
\[
\omega_R(t)\leq 1520\times 20^{n(t/20)}t.
\]
\end{theorem}

\begin{proof}[Proof of Theorem \ref{thm:1sec3}] As in \cite{M} in order to find a $\alpha$-H\"older retraction from $X$ onto $K$, we first note that for $r_n = 20^{\frac{n}{\alpha -1}}$, $n(t)\leq \log_{20}( t^{\alpha -1}) +1$. Hence 
\[
n(t/20)\leq \log_{20}(t^{\alpha -1}) + 2- \alpha,\quad t\in\mathbb{R}^+.
\]
This latter inequality combined with Theorem \ref{thm:2sec3} implies
\[
\begin{split}
\omega_R(t)&\leq 1520\times 20^{n(t/20)}\times t\leq 1520\times 20^{2-\alpha}\times t^\alpha,
\end{split}
\]
proving that $R$ is $\alpha$-H\"older $1520\times 20^{2-\alpha}$-Lipschitz, and finishing the proof.
\end{proof}



\section{Main fixed point free results}\label{sec:4}

\begin{theorem}\label{thm:M0sec4} Let $X$ be a Banach space. Then there exists $K\in\mathcal{B}(B_X)$ such that for any $\varepsilon\in (0,1)$, $K$ fails the FPP for $(1+\varepsilon)$-Lipschitz maps with null minimal displacements. If $X$ contains a complemented basic sequence then $K$ can be taken to be $B_X$. 
\end{theorem}

\begin{proof} Let $(x_n)_{n=1}^\infty$ be a normalized basic sequence in $X$. Up to taking an equivalent renorming on $\llbracket x_n\rrbracket$, we may assume that $(x_n)_{n=1}^\infty$ is normalized and premonotone, that is, $\|\sum_{n=N}^\infty t_n x_n\|\leq \|x\|$ for all $x=\sum_{n=1}^\infty t_n x_n \in \llbracket x_n\rrbracket$ and $N\in\mathbb{N}$. Fix a decreasing sequence $\beta_n\to 0$ in $(0,1)$ so that $ \sum_{n=1}^\infty \beta_n \leq 1$. For $n\in\mathbb{N}$, set $\alpha_n = 1 - \beta_n$. Notice that $\| \sum_{n=1}^\infty \alpha_n t_n x_n\|\leq \| x\|$ for all $x=\sum_{n=1}^\infty t_n x_n\in \llbracket x_n \rrbracket$. Set $K=B_{\llbracket x_n \rrbracket}$ and define $F\colon K\to K$ by 
\[
F(x) = \sum_{n=1}^\infty \alpha_n t_n x_n + (1 - \|x\|)\sum_{n=1}^\infty \beta_n x_n,\quad  x=\sum_{n=1}^\infty t_n x_n\in K.
\]
Then $F$ is Lipschitz and fixed-point free. Notice also that $\| x_k - F(x_k)\| = \beta_k$ for all $k\in\mathbb{N}$, so $\mathrm{d}(F, K)=0$. In addition, if $\lambda \approx 1$ then $T=(1-\lambda)F + \lambda I$ is a $(1+\varepsilon)$-Lipschitz mapping with $\digamma(T)=\emptyset$ and $\mathrm{d}(T, K)=0$. 
\end{proof}

\smallskip 

\noindent Our next result shows that shift domination improves mapping regularity. 

\smallskip 

\begin{theorem}\label{thm:M1sec5} Let $X$ be a Banach space. Assume that $(x_n)_{n=1}^\infty$ is a semi-normalized unconditional shift basic sequence. Then the following hold:
\begin{itemize}\setlength\itemsep{1.5mm}
\item[(i)] If $(x_n)_{n=1}^\infty$ is equivalent to the unit basis of $\co$, then for any  $\varepsilon\in (0,1)$ there is $K\in\mathcal{B}(B_X)$ which fails the FPP for uniformly asymptotically regular and uniformly $(1+\varepsilon)$-Lipschitz mappings. 
\item[(ii)] If $(x_n)_{n=1}^\infty$ is not equivalent to the unit basis of $\co$, then there is $K\in\mathcal{B}(B_X)$ which fails the FPP for asymptotically regular Lipschitz mappings. 
\item[(iii)] If $\llbracket x_n\rrbracket$ is uniformly convex and $(x_n)_{n=1}^\infty$ is $1$-unconditional shift, then there is $K\in\mathcal{B}(B_X)$ which fails the FPP for uniformly asymptotically regular Lipschitz mappings. 
\end{itemize}
\end{theorem}

\begin{proof} (i) By \cite[Lemma 2.2]{RCJ}, taking a suitable block basis of $(x_n)_{n=1}^\infty$, if needed, we may assume that $(x_n)_{n=1}^\infty\subset B_X$ and is $((1-\varepsilon/2)^{-1}, 1)$-equivalent to the unit basis of $\co$. Set
\[
K=\Bigg\{ \sum_{n=1}^\infty t_n x_n\, \colon\, 0\leq t_n\leq 1\;\forall n\in\mathbb{N}\Bigg\}. 
\]
Notice that $K\in B_X$. We shall now follow \cite{ACM} (see also \cite{Bar1,BarG}) to define a fixed-point free map $F$ on $K$ with desired properties. Fix an increasing sequence $\alpha_n\to 1$ in $(0,1)$. For $n\in\mathbb{N}$, set $\beta_n = 1 - \alpha_n$. Define $F\colon K\to K$ by
\[
F\Bigg( \sum_{n=1}^\infty t_n x_n\Bigg) = \sum_{n=1}^\infty t_n \alpha_n x_n + \sum_{n=1}^\infty \beta_n x_n.
\]
It is clear that $F$ is affine and fixed-point free. Furthermore, note that for fixed $x= \sum_{n=1}^\infty t_n x_n$, an easy induction procedure yields
\[
F^m(x) = \sum_{n=1}^\infty t_n \alpha^m_n x_n + \sum_{n=1}^\infty \beta_n \sum_{i=0}^{m-1}\alpha^i_n x_n\quad(m\geq 1). 
\]
It follows from this that $F$ is uniformly $(1+\varepsilon)$-Lipschitz. Moreover, observe that  $\| F^{m+1}(x) - F^m(x)\| \leq 2\sup_{n\in\mathbb{N}}\alpha_n^m(1- \alpha_n)$ for all $m\geq 1$, which implies $F$ is uniformly asymptotically regular. The proof of (i) is complete.

\vskip .2cm 

\noindent The proofs of (ii) and (iii) are based on Lin's ideas used in \cite{Lin}, however, in view of our next results, some technical modifications are needed. We may assume that $(x_n)_{n=1}^\infty \subset B_X$. Set
\[
K=\Bigg\{ x=\sum_{n=1}^\infty t_n x_n\, \colon\, t_n\geq 0\; \forall n\in\mathbb{N},\,\text{ and }\, \| x\|\leq 1\Bigg\}. 
\]
It is easily seen that $K\in \mathcal{B}(B_X)$. For $x\in K$  define
\begin{equation}\label{eqn:1Sec5}
g(x) = \max(|x|, 1- \|x\|)x_1 + \sum_{n=1}^\infty \max(t_n, t_{n +1}) x_{n+1},
\end{equation}
where $|x|:=\max\{ t_n\colon n\in\mathbb{N}\}$. Next consider the mapping $F\colon K\to K$ given by 
\[
F(x) = \frac{1 }{\| g(x)\|} g(x),\quad x\in K.
\]
Notice that $F$ is fixed-point free. Indeed, towards a contradiction, assume that $x = F(x)$ for some point $x= \sum_{n=1}^\infty t_n x_n$ in $K$. Then $\| x\|=1$ and hence we have
\[
\sum_{n=1}^\infty t_n x_n = \frac{1}{\|g(x)\|}\Big( |x| x_1 + \sum_{n=1}^\infty \max(t_n,t_{n+1}) x_{n+1}\Big).
\]
It follows therefore that
\[
t_1 = \frac{1}{\|g(x)\|} |x|,\,\, t_2 = \frac{1}{\| g(x)\|}\max(t_1,t_2),\,\, t_3 = \frac{1}{\| g(x)\|}\max(t_2, t_3),\dots
\]
Then $|x|=0$. Indeed, since $(x_n)_{n=1}^\infty$ is basic and semi-normalized, $(t_n)_{n=1}^\infty\in \co$. So, we may find an integer $\sigma\in\mathbb{N}$ so that $|x|= t_\sigma$. It follows then from the equalities above that
\[
t_\sigma= \frac{1}{\| g(X)\|} t_\sigma,\,\,t_{\sigma+1}= \frac{1}{\|g(x)\|}t_\sigma,\,\, t_{\sigma+2}= \frac{1}{\| g(x)\|} t_\sigma,\dots
\]
Consequently, $t_n= t_\sigma$ for all $n\geq \sigma$. This certainly implies $t_\sigma=0$, otherwise we would reach at a contradiction with the fact that $(t_n)_{n=1}^\infty \in \co$. Thus $|x|=0$ and hance $x=0$, contradicting the equality $\| x\|=1$. Therefore $F$ is fixed-point free.



\smallskip 
\noindent In what follows we shall use $K$ and $F$ to prove (ii) and (iii). Let us first prove (ii). We may assume (after passing to an equivalent norm) that $(x_n)_{n=1}^\infty$ is $1$-unconditional shift with $(x_{n+1})_{n=1}^\infty \lesssim_D (x_n)_{n=1}^\infty$ for some $D\geq 1$. Then $g$ is Lipschitzian. Indeed, take any $x = \sum_{n=1}^\infty t_n x_n$ and $y=\sum_{n=1}^\infty s_n x_n$ in $K$. For $n\in\mathbb{N}$, set $a_n= t_n - s_n$ and $\tilde{a}_n=\max(t_n, t_{n+1}) - \max(s_n, s_{n+1})$. Then
\[
g(x) - g(y) = \big(\max(|x|, 1 - \| x\|) - \max(|y|, 1 - \|y\|)\big)x_1 + \sum_{n=1}^\infty \tilde{a}_n x_{n+1},
\]
from which it follows that 
\[
\begin{split}
\| g(x) - g(y)\| &\leq \big| \max(|x|, 1 - \| x\|) - \max(|y|, 1 - \|y\|)\big| \| x_1\| + \Bigg\| \sum_{n=1}^\infty |\tilde{a}_n| x_{n+1}\Bigg\|\\[1.5mm]
&\leq \big( ||x| - |y|| + \| x - y\|\big) \| x_1\| + \Bigg\| \sum_{n=1}^\infty \big(|a_n| + |a_{n+1}|\big) x_{n+1}\Bigg\|\\[1.5mm]
&\leq \max_n|a_n| \| x_1\| + \| x - y\|\| x_1\| + \Bigg\| \sum_{n=1}^\infty |a_n|x_{n+1}\Bigg\| + \Bigg\| \sum_{n=1}^\infty |a_{n+1}|x_{n+1}\Bigg\| \\[1.5mm]
&\leq \frac{1}{\inf_n \|x_n\|} \max_n \| a_n x_n\|  + (D+2)\| x - y\| \\[1.5mm]
&\leq \Big(\frac{1}{\inf_n\| x_n\|}+ D+2\Big)\| x - y\|,
\end{split}
\]
where in the above inequalities we used the fact that $\| x_1\|\leq 1$ together with the $1$-suppression unconditionality and the shift property of $(x_n)_{n=1}^\infty$. 

\vskip .2cm 
\noindent On the other hand, using again the $1$-unconditionality property of $(x_n)_{n=1}^\infty$, we get
\begin{equation}\label{eqn:2Sec5}
\| g(x)\|\geq \| x \|  
\end{equation}
from which it easily follows that
\begin{equation}\label{eqn:3Sec5}
\| g(x)\| \geq \frac{\| x_1\|}{2}\quad\text{for all } x\in K.
\end{equation}
Combining these facts we can also deduce that $F$ is Lipschitz. Indeed, to see this fix arbitrary $x, y$ in $K$. Using triangle inequality, (\ref{eqn:3Sec5}) and the Lipschitz property of $g$, we have
\[
\begin{split}
\| F(x) - F(y)\| &\leq \frac{1}{\| g(x)\|} \|g(x) - g(y)\| +  \Big\| g(y)\Big( \frac{1}{\|g(x)\|} - \frac{1}{\|g(y)\|}\Big)\Big\|\\[1.5mm]
&\leq \frac{2}{\|g(x)\|}\| g(x) - g(y)\|\leq \frac{4}{\| x_1\|}\Big(\frac{1}{\inf_n\| x_n\|}+ D+2\Big) \| x - y\|. 
\end{split}
\]
To finishes the proof of (ii) it remains to show that $F$ is asymptotically regular, i.e.
\[
\lim_{n\to \infty} \| F^{n+1}(x) - F^n(x)\|=0\quad\text{for all } x\in K. 
\]
Following \cite{Lin} we first observe the following easily verified facts that stem from the $1$-unconditionality property of $(x_n)_{n=1}^\infty$.

\smallskip 

\noindent{\bf Fact 1.} If $x\in K$ then $\| F(x)\|=1$.

\vskip .15cm 
\noindent{\bf Fact 2.} If $F^{n+1}(x) = \sum_{i=1}^\infty a_i x_i$ then  $a_1=a_2=\dots =a_n\leq 1/\Phi(n)$.
\vskip .15cm 
\noindent{\bf Fact 3.} If $x= \sum_{i=1}^\infty a_i x_i \in K$ and $g(x) = \sum_{i=1}^\infty b_i x_i$ then $a_n \leq b_n$ for all $n\in \mathbb{N}$. 

\vskip .15cm 
\noindent{\bf Fact 4.} $2\| x\|\leq  \| g(x) + x\|$ for all $x\in K$. 

\vskip .15cm 
\noindent{\bf Fact 5.} $\|g(F^{n+1}(x))\|\leq \frac{1}{\Phi(n)} +  1$ for all $n\in\mathbb{N}$ and $x\in K$. 
 
\smallskip 

\noindent Now fix $x\in K$. Notice that $g(F^{n+1}(x))= F^{n+2}(x) \| g(F^{n+1}(x))\|$. Moreover, by (\ref{eqn:2Sec5}), $\|g(F^{n+1}(x))\|\geq 1$. Last but not least, we note that for each $n$ the vector $g(F^{n+1}(x)) - F^{n+1}(x)$ is the tail of a convergent series. So, combining all these facts and using triangle inequality, we obtain 
\[
\begin{split}
\| F^{n+2}(x) - F^{n+1}(x)\|&\leq \|F^{n+2}(x) - g(F^{n+1}(x))\| + \| g(F^{n+1}(x)) - F^{n+1}(x)\|\\[1.4mm]
&\leq \|g(F^{n+1}(x))\| - 1 + \| g(F^{n+1}(x)) - F^{n+1}(x)\|\\[1.4mm]
&\leq \frac{1}{\Phi(n)} + \| g(F^{n+1}(x)) - F^{n+1}(x)\|\to 0,\quad \text{as } n\to\infty. 
\end{split}
\]
We now prove (iii). To this end, all we need to show is that
\[
\lim_{n\to \infty} \| F^{n+1}(x) - F^n(x)\|=0\quad\text{uniformly on } x.
\]
This in turn follows form the uniform convexity of $\llbracket x_n\rrbracket$ and previous mentioned facts. Indeed, since $X$ is uniformly convex, for any $\epsilon>0$ there exists $\delta>0$ such that if $\| x\|\leq 1$, $\|y\|\leq 1$ and $\|x - y\|> \delta$, then $\| x + y\|/2< 1 - \epsilon$. Also, recall that $(x_n)_{n=1}^\infty$ is $1$-unconditional shift. Hence for $1/\Phi(n) < \epsilon$, we have
\[
\begin{split}
1 + \epsilon \geq 1 + \frac{1}{\Phi(n)}&\geq \| g(F^{n+1}(x))\|\\[1.2mm]
&\geq \frac{1}{2}\| g(F^{n+1}(x)) + F^{n+1}(x)\|\geq 1. \quad(\text{by {\bf Fact 4}})
\end{split}
\]
Consequently,
\[
\| g(F^{n+1}(x)) - F^{n+1}(x)\|< \delta\times (1 + 1/\Phi(n))
\]
and hence we deduce
\[
\begin{split}
\| F^{n+2}(x) - F^{n+1}(x)\|&\leq \|F^{n+2}(x) - g(F^{n+1}(x))\| + \| g(F^{n+1}(x)) - F^{n+1}(x)\|\\[1.2mm]
&\leq \|g(F^{n+1}(x))\| - 1 + \delta\times \Big( 1 + \frac{ 1}{\Phi(n)}\Big)\\[1.2mm]
&\leq \frac{1}{\Phi(n)} + \delta\times \Big(1 + \frac{1}{\Phi(n)}\Big).
\end{split}
\]
It follows from this that $F$ is uniformly asymptotically regular, proving (iii).
\end{proof}

\vskip .1cm  





\noindent Our next result requires the following lemma. 

\vskip .1cm

\begin{lemma}\label{lem:1sec5} Let $X$ be a Banach space having a semi-normalized unconditional basis $(x_n)_{n=1}^\infty\subset B_X$. Then $K$ is a Lipschitz retract of $B_X$, where
\[
\Bigg\{ x=\sum_{n=1}^\infty t_n x_n\, \colon\, t_n\geq 0\; \forall n\in\mathbb{N},\,\text{ and }\, \| x\|\leq 1\Bigg\}. 
\]
Furthermore, if $(x_n)_{n=1}^\infty$ is $1$-unconditional then the retraction is $1$-Lipschitz. 
\end{lemma}

\begin{proof} We may assume without loss of generality that $(x_n)_{n=1}^\infty$ is $1$-unconditional. Define a mapping $R\colon B_X\to K$ by $R(x) = \sum_{n=1}^\infty |x^*_n(x)| x_n$, where $(x^*_n)_{n=1}^\infty$ denotes the associated biorthogonal functionals. Then $R$ is a well-defined retraction onto $K$. In addition, the $1$-unconditionality property of the basis readily implies $R$ is $1$-Lipschitz. This finishes the proof. 
\end{proof}

\smallskip 

\begin{theorem}\label{thm:M2sec5} Let $X$ be a Banach space with a semi-normalized subsymmetric basis $(x_n)_{n=1}^\infty$. Assume that $X$ does not contain any isomorphic copies of $\co$.  Then,
\begin{itemize}
\item[(i)] $B_X$ fails the FPP for asymptotically regular Lipschitz maps.
\item[(ii)] If $X$ is uniformly convex, then $B_X$ fails the FPP for uniformly asymptotically regular Lipschitz maps. 
\end{itemize}
\end{theorem}

\begin{proof} Assume without loss of generality that $(x_n)_{n=1}^\infty$ is normalized. Take $K$ and $F$ to be the set and the mapping given in the proof of Theorem \ref{thm:M1sec5} (ii). By Lemma \ref{lem:1sec5} there exists a Lipschitz retraction $R\colon B_X \to K$. In order to check (i) it suffices to note that the composition $T:=F\circ R$ plainly defines an asymptotically regular Lipschitz mapping with $\digamma(T)=\emptyset$. 

\vskip .1cm 
\noindent Let us prove (ii). The idea is to consider an equivalent norm on $X$ so that under the new norm, $(x_n)_{n=1}^\infty$ becomes $1$-unconditional shift and $X$ remains uniformly convex. To that effect, we start with by taking an equivalent norm $|\cdot|$ on $X$ so that $(x_n)_{n=1}^\infty$ is $1$-subsymmetric (cf. \cite[Theorem 3.7]{Anso}). Using then a result of N.I and V.I Gurariy \cite{GG} we deduce that $X$, {\it a fortiori} $(X,|\cdot|)$, does not contain $\ell^n_\infty$ uniformly for large $n$, nor does it contain $\ell_1^n$ uniformly for large $n$ either. By a result of Figiel and Johnson (cf. \cite[Lemma 3.1, Remarks 3.1--3.2, and Theorem 3.1]{FJ}) $X$ can then be equivalently renormed to be uniform convex with $(x_n)_{n=1}^\infty$ being $1$-unconditional and $1$-subsymmetric. By the first part of the proof, the result follows. 
\end{proof}

\smallskip 

\noindent Our third result also requires a retraction lemmata and some propositions. 

\smallskip 

\begin{lemma}\label{lem:2sec5} Let $X$ be a Banach space that does not contain a subspace isomorphic to $\ell_1$. Assume that $X$ contains an isomorphic copy of $\co$. Then for any $\varepsilon\in (0,1)$, there exist a complemented basic sequence $(x_n)_{n=1}^\infty$ in $B_X$ equivalent to the unit basis of $\co$ and a $(1+\varepsilon)$-Lipschitz retraction $R\colon B_X \to K$, where
\[
K=\Bigg\{ x=\sum_{n=1}^\infty t_n x_n\, \colon\, 0\leq t_n\leq 1\;\forall n\in\mathbb{N}\Bigg\}. 
\]
\end{lemma}

\begin{proof} Pick $\eta\in (0,1)$ so that 
\[
\frac{1+\eta}{ 1-\eta}< 1 + \varepsilon.
\] 
The proof of \cite[Theorem 4.3]{GP} (cf. also \cite[Proposition 1]{DRT} and \cite[Theorem 1]{JR}) yields a basic sequence $(x_n)_{n=1}^\infty$ in $B_X$ $((1-\eta)^{-1}, 1)$-equivalent to the unit basis of $\co$, and a projection $P\colon X\to \llbracket x_n\rrbracket$ with $\|P\|< 1+\eta$. Then the mapping  $Q\colon \llbracket x_n\rrbracket \to K$ given by $Q(x) = \sum_{n=1}^\infty \min(1, |t_n|)x_n$ is a well-defined $(1-\eta)^{-1}$-Lipschitz retraction onto $K$. Thus $R=Q \circ P$ defines a $1+\varepsilon$-Lipschitz retraction onto $K$, proving the lemma.  
\end{proof}


\begin{proposition}\label{prop:1sec5} Let $X$ be a Banach space and suppose that $K\in\mathcal{B}(B_X)$ is a Lipschitz retract of $B_X$. Then (i) if $K$ fails the FPP for uniformly Lipschitz maps with null minimal displacement, then the same happens for $B_X$; (ii) if $K$ fails the FPP for Lipschitz maps with null minimal displacement, then for any $\varepsilon\in (0,1)$ there exists a $(1+\varepsilon)$-Lipschitz mapping $T\colon B_X\to B_X$ with $\digamma(T)=\emptyset$ and $\mathrm{d}(T, B_X)=0$.
\end{proposition}

\begin{proof} Let $F\colon K\to K$ be a (uniformly) Lipschitz mapping such that $\digamma(F)=\emptyset$ and $\mathrm{d}(F,K)=0$, and let $R\colon B_X\to K$ be a Lipschitz retraction. Set $G=F\circ R$. Clearly $\digamma(G)=\emptyset$ and $\mathrm{d}(G,B_X)=0$. If $F$ is uniformly Lipschitz, so is $G$. Let's check (ii). To this end, let $L$ denote the Lipschitz constant of $G$ and pick $\lambda>0$ so close to $1$ as to satisfy $(1- \lambda)L + \lambda < 1 + \varepsilon$. Finally, define $T= (1-\lambda)G + \lambda I$. It readily follows that $T$ is $(1+\varepsilon)$-Lipschitz, $\digamma(T)=\emptyset$ and $\mathrm{d}(T, B_X)\leq (1-\lambda)\mathrm{d}(F,K)$, and this proves the result. 
\end{proof}

\vskip .1cm 

\begin{proposition}\label{prop:3sec5} Let $X$ be a Banach space that contains a complemented copy of $\ell_1$. Then there exists $K\in\mathcal{B}(B_X)$ which is a Lipschitz retract of $B_X$ and fails the FPP for uniformly Lipschitz maps with null minimal displacement. 
\end{proposition}

\begin{proof} By assumption there is a basic sequence $(x_n)_{n=1}^\infty$ which is $(A^{-1},1)$-equivalent to the unit basis of $\ell_1$, $A>0$. In addition, there is a projection $P\colon X\to \llbracket x_n\rrbracket $. Set
\[
K=\Bigg\{ \sum_{n=1}^\infty t_n x_n \,\colon\, t_n\geq 0\;\forall n\in\mathbb{N},\,\sum_{n=1}^\infty t_n=1\Bigg\}. 
\]
Clearly $K\in\mathcal{B}(B_X)$ and fails the FPP for affine uniformly Lipschitz maps. It is easy to see that the positve cone $\llbracket x_n\rrbracket^+$ is a $A^{-1}$-Lipschitz retract of $\llbracket x_n\rrbracket$. Repeating the arguments from the proof in \cite[Lemma 3.1]{BenJap} we can build a Lipschitz retraction from $\llbracket x_n\rrbracket^+$ onto $K$. By composing these maps we get the result. 
\end{proof}

\vskip .1cm 

\begin{remark} Recall that a basis is called {\it spreading} if it is equivalent to all of its subsequences. 
\end{remark}

\noindent The following result is known, but for completeness we include its proof here. 

\vskip .1cm 

\begin{proposition}\label{prop:9sec4} Let $X$ be a Banach space with a spreading Schauder basis. If $X$ contains a subspace isomorphic to $\ell_1$ then it contains a complemented copy of $\ell_1$.
\end{proposition}

\begin{proof} If the basis of $X$ is unconditional then the result follows directly from \cite[Theorem 1]{FW}. On the other hand, if it is conditional then by \cite[Proposition 8.7]{AMS}, $X$ contains a complemented copy of $\ell_1$.
\end{proof}

\vskip .1cm 

\begin{proposition}\label{prop:10sec4} Let $X$ be a Banach space. Assume that $B_X$ fails the FPP for (uniformly) Lipschitz maps with null minimal displacement. Then for any $\alpha\in (0,1)$ and $\lambda>0$ there exists a fixed-point free (uniformly) $\alpha$-H\"older $\lambda$-Lipschitz mapping $T\colon B_X \to B_X$ with $\mathrm{d}(T, B_X)=0$. 
\end{proposition}

\begin{proof} Let $S\colon B_X\to B_X$ be a (uniformly) $L$-Lipschitz map with $\digamma(S)=\emptyset$ and $\mathrm{d}(S, B_X)=0$. For $r>0$ with $2L r^{1-\alpha} \leq \lambda$ define $S_r\colon B_X(r)\to B_X(r)$ by $S_r(x) =r S(x/r)$. Then $S_r$ is (uniformly) $(2r)^{1-\alpha} L$-Lipschitz and satisfies $\digamma(S_r)=\emptyset$ and $\mathrm{d}(S_r, B_X(r))=0$. Finally, taking a $2$-Lipschitz retraction $R\colon B_X \to B_X(r)$, the map $T(x) = S_r(Rx)$ has the desired properties. 
\end{proof}

\vskip .1cm 

\begin{theorem}\label{thm:M3sec5} Let $X$ be a Banach space with a spreading Schauder basis. Then,
\begin{itemize}
\item[(i)] for any $\varepsilon\in (0,1)$, $B_X$ fails the FPP for $(1 +\varepsilon)$-Lipschitz maps with null minimal displacement.
\item[(ii)] For any $\alpha\in (0, 1)$ and $\lambda> 0$, $B_X$ fails the FPP for $\alpha$-H\"older $\lambda$-Lipschitz maps with null minimal displacement. 
\end{itemize}
In the case $X$ contains a copy of $\co$, $B_X$ fails the FPP for uniformly Lipschitz maps with null minimal displacement. In addition, the map built in (ii) is  uniformly $\alpha$-H\"older $\lambda$-Lipschitz.
\end{theorem}

\begin{proof} We distinguish into two cases:

\smallskip 
\noindent{\bf Case 1.} $X$ contains an isomorphic copy of $\co$. First assume that $X$ does not contain a subspace that is isomorphic to $\ell_1$. Then the proof of (i) is an easy consequence of Theorem \ref{thm:M1sec5}-(i), Lemma \ref{lem:2sec5} and Proposition \ref{prop:1sec5}. Now assume that $X$ contains an isomorphic copy of $\ell_1$. By Proposition \ref{prop:9sec4}, $\ell_1$ is complemented in $X$. Thus by Propositions \ref{prop:1sec5} and \ref{prop:3sec5} (i) follows. As for (ii), separability implies $X$ is separably Sobczyk. By \cite[Theorem 5.1-(3)]{Bar} (ii) follows. 

\vskip .1cm 
\noindent{\bf Case 2.} $X$ does not contain subspaces isomorphic to $\co$. The proof of (ii) follows immediately from (i) and Proposition \ref{prop:10sec4}. So, we only need to check (i). As the basis is spreading, it is bounded (cf. proof of \cite[Lemma 2.5]{Anso}). Let us distinguish between two subcases according to whether or not it is unconditional. If the basis is unconditional then it is, by definition, subsymmetric. By Theorem \ref{thm:M2sec5}-(i), there exists an asymptotically regular Lipschitz mapping $T\colon B_X\to B_X$ without fixed points. If $\lambda \approx 1$ then $T_\lambda=(1-\lambda)T+ \lambda I$ fulfills (i). Now assume that the basis of $X$ is conditional. By a result of Freeman, Odell, Sari and Zheng \cite{FOSZ}, $X$ contains a complemented subspace $U$ that has a semi-normalized subsymmetric Schauder basis. By Theorem \ref{thm:M2sec5} $B_U$ fails the FPP for asmptotically regular Lipschitz maps. Composing with projection and applying once more shrinking's Lipschitz-constant argument, we get the desired result. 
\end{proof}

\smallskip 

\noindent In our next results, we solve the unit ball version of questions ($\mathcal{Q}1$) and ($\mathcal{Q}3$) in Hilbert spaces and $L_p$ spaces with $1< p<\infty$ (cf. also \cite[Open questions, p.16]{Bar}). 

\smallskip 

\begin{proposition}\label{prop:12sec4} Let $H$ be an infinite dimensional Hilbert space. Then:
\begin{itemize}
\item[(i)] There exists a uniformly asymptotically regular Lipschitz mapping $T\colon B_H\to B_H$ with no fixed points. 
\item[(ii)] For any $\alpha\in (0,1)$ and $\lambda>0$, $B_H$ fails the FPP for $\alpha$-H\"older $\lambda$-contractive maps with null minimal displacement. 
\end{itemize}
\end{proposition}

\begin{proof} It is well-known (and simple to check) that $H$ contains a linear complemented isometric copy of $\ell_2$. If $R$ is a nonexpansive projection onto $\ell_2$, as its isometric copy, then after taking suitable compositions, it suffices to show (i) for $H=\ell_2$. Let $(e_i)_{i=1}^\infty$ denote the unit basis of $\ell_2$ and take the set $K\in \mathcal{B}(B_{\ell_2})$ considered by Lin in \cite{Lin}. He showed that $K$ fails the FPP for uniformly asymptotically regular Lipschitz maps. Let $F\colon K\to K$ denote such a map. Next pick a nonexpansive projection $P\colon \ell_2 \to K$. Then the composition $T=F\circ P$ maps $B_{\ell_2}$ into itself, has no fixed points, and is Lipschitz and uniformly asymptotically regular. This proves (i). The proof of (ii) is virtually the same as the proof of \cite[Proposition 4.1-(9)]{Bar}.
\end{proof}

\smallskip 

\begin{corollary}\label{cor:13sec4} Let $1< p< \infty$. Then:
\begin{itemize}
\item[(i)]  There exists a uniformly asymptotically regular Lipschitz mapping $T\colon B_{L_p}\to B_{L_p}$ with no fixed points. 
\item[(ii)] For any $\alpha\in (0,1)$, $B_{L_p}$ fails the FPP for $\alpha$-H\"older $\lambda$-contractive maps. 
\end{itemize}
\end{corollary}

\begin{proof} By \cite[6.8, p.163]{AK} $\ell_2$ is linearly isometric to a complemented subspace $X$ of $L_p$. Combining the arguments in Proposition \ref{prop:12sec4} with the arguments contained in the proof of \cite[Proposition 4.1-(9)]{Bar}, the result follows.
\end{proof}

\smallskip 

\noindent Retraction's Theorem \ref{thm:1sec3} provides the following abstract fixed-point free result.

\smallskip 

\begin{theorem}\label{thm:4sec3}
Let $X$ be a normed space. Assume that for any $\theta\in (0,1)$ and any $\mu\in (0,1)$, there exist a $(20^{\frac{n}{\theta -1}})$-flat set $K\in\mathcal{B}(B_X(\mu))$ and a fixed-point free (uniformly) $\omega$-Lipschitz mapping $F\colon K\to K$ with null minimal displacement. Then for any $\alpha\in (0,1)$ and $\lambda>0$ there exists a fixed-point free mapping $T\colon B_X\to B_X$ such that $\mathrm{d}(T, B_X)=0$,
\begin{itemize}
\item[(i)] $\|T(x) - T(y)\|\leq \omega_F(\lambda \| x - y\|^\alpha)$ for all $x, y\in B_X$; and (in the uniform case),
\item[(ii)] $\|T^n(x) - T^n(y)\|\leq \sup_{n\in\mathbb{N}}\omega_{F^n}(\lambda \| x - y\|^\alpha)$ for all $x, y\in B_X$ and $n\in\mathbb{N}$.
\end{itemize}
\end{theorem}

\begin{proof} Fix $\alpha<\theta <1$ and set $\gamma = \alpha/\theta$. Next choose a number $\mu>0$ so that $(2\mu)^{1-\gamma}\times 1520\times 20^{2-\theta}\leq \lambda$. By assumption there exist a $(20^{\frac{n}{\theta -1}})$-flat set $K\in\mathcal{B}(B_X(\mu))$ and a fixed-point free mapping $F\colon K\to K$ with $\mathrm{d}(F, K)=0$. By Theorem \ref{thm:1sec3} we find a $\theta$-H\"older Lipschitz retraction $R\colon X\to K$ whose H\"older-Lipschitz constant is $1520\times 20^{2-\theta}$. Define $T:= F\circ (R|_{B_X})$. As $\mathrm{d}(T, B_X)\leq \mathrm{d}(T,K)$ and $T|_K\equiv F$, $\mathrm{d}(T,B_X)=0$. Note that $T$ is fixed-point free, otherwise $\digamma(T)\subseteq K$ and $x=Tx$ would imply $x = Fx$, contradicting the fact that $\digamma(F)=\emptyset$. 

\vskip .1cm 
\noindent Let's prove (i)--(ii). Assume first $F$ is $\omega$-Lipschitz. Fix $x, y\in B_X$. Then
\[
\begin{split}
\|T(x) - T(y)\| &\leq \omega_F( \|R(x) - R(y)\|)\\[1.5mm]
&\leq \omega_F( \|R(x) \| + \|R(y)\|)^{1 -\gamma} \|R(x) - R(y)\|^\gamma)\\[1.5mm]
&\leq \omega_F((2\mu)^{1-\gamma}\times ( 1520\times 20^{2- \theta})^\gamma \|x - y\|^{\gamma\theta})\\[1.5mm]
&\leq \omega_F(\lambda \|x- y\|^\alpha).
\end{split}
\]
Assume now that $F$ is uniformly $\omega$-Lipschitz. By induction, suppose that we have already proved for some $n\in \mathbb{N}$ that both $T^n(x)$ and $T^n(y)$ belongs to $K$. Thus 
\[
\begin{split}
\|T^{n+1}(x) - T^{n+1}(y)\|&=\| F^{n+1}R(x) - F^{n+1}R(y)\|\\[1.5mm]
&\leq \omega_{F^n}(\lambda \|x  - y\|^\alpha).     
\end{split}
\]
This completes the induction step and finishes the proof of the theorem.
\end{proof}

\medskip 

\noindent Our last result reads.

\newpage

\begin{theorem}\label{thm:M4sec5} Let $X$ be a Banach space. Then the following hold:
\begin{itemize}
\item[(i)] For any decreasing null sequence of positive numbers $(r_n)_{n=1}^\infty$ with $r_{n +k}\leq r_n\cdot r_k$ for $n, k\in\mathbb{N}$, there exist a $(r_n)$-flat set $K\in \mathcal{B}(B_X)$ and a fixed-point free affine mapping $F\colon K\to K$ such that, for all $x, y\in K$ and $n\in\mathbb{N}$,
\[
\| F^n(x) - F^n(y)\| \leq r_n\big( \| x - y\| +1).
\]
\item[(ii)] For any $\alpha\in (0,1)$ there exists a fixed-point free mapping $T\colon B_X \to B_X$ with null minimal displacement such that, for all $x, y\in B_X$ and $n\in\mathbb{N}$, 
\[
\| T^n(x) - T^n(y)\|\leq 20^{\frac{n}{\alpha-1}} \big( \| x - y\|^\alpha +1\big).
\]
\end{itemize}
\end{theorem}

\begin{proof} (i) Let $(x_n)_{n=1}^\infty$ be a normalized basic sequence in $X$ and $\mathcal{K}$ denote its basic constant. Pick a decreasing null sequence of positive numbers $(\alpha_n)_{n=1}^\infty$ so that
\begin{equation}\label{eqn:5sec5}
\max\Big(3\sum_{i=n}^\infty \alpha_i,\frac{2\mathcal{K}}{1520\times 20}\sum_{i=1}^\infty \frac{\alpha_{i+n}}{\alpha_i}\Big)\leq \min(1, r_{n+1})\;\;\forall\,n\in\mathbb{N}.
\end{equation}
We now define $w_n = \sum_{i=1}^n \alpha_i x_i$. The sequence $(w_n)_{n=1}^\infty$ can be seen as a weighted summing basis of $(x_n)_{n=1}^\infty$. It is easy to see that this sequence strongly converges to $w_0=\sum_{i=1}^\infty \alpha_i x_i$. Since $w_0$ is not null, it cannot be basic. Notice however that
\[
\frac{\alpha_1}{2\mathcal{K}}\leq \| w_n\|\leq \sum_{i=1}^\infty \alpha_i\;\;\forall\, n\in\mathbb{N}.
\]
Moreover defining 
\[
w^*_n = \frac{x^*_n}{\alpha_n} - \frac{x^*_{n+1}}{\alpha_{n+1}},
\]
one can easily verify that $\{ w_n; w^*_n\}_{n=1}^\infty$ is a biorthogonal system on $\llbracket x_n\rrbracket$. Set
\[
K= \Bigg\{ t_0w_0 +\sum_{n=1}^\infty t_n w_n \,\colon \, t_n\geq 0\, \forall n\in\mathbb{N}\cup\{0\}\;\text{ and }\; \sum_{n=0}^\infty t_n \leq \mu\Bigg\}.
\]

\noindent{\bf Claim 1.} $K\subset B_X(\mu)$ is closed. Assume that $u^k =t^k_0 w_0 + \sum_{n=1}^\infty t^k_n w_n\in K$ for all $k\in\mathbb{N}$ and $\| u^k - u\|\to 0$ for some $u\in X$. We may write
\[
u^k = \sum_{n=1}^\infty \Bigg(t^k_0 + \sum_{i=n}^\infty t^k_i\Bigg) \alpha_n x_n.
\]
Then $u\in \llbracket \alpha_n x_n\rrbracket$. Since $(\alpha_n x_n)_{n=1}^\infty$ is basic, there exist unique scalars $(t_n)_{n=1}^\infty$ so that $u = \sum_{n=1}^\infty t_n \alpha_n x_n$. For each $n\in\mathbb{N}$, $w^*_n(u^k) \to w^*_n(u)$ and $x^*_n(u^k)\to x^*_n(u)$ as $k\to\infty$. This certainly implies $(t_n)_{n=1}^\infty$ is a non-increasing sequence in $[0,\mu]$. Let $t_0=\lim_{n\to\infty}t_n$. Now an easy manipulation using Abel's summation yields that $u= t_0 w_0 + \sum_{n=1}^\infty (t_n - t_{n+1})w_n$. Thus $u\in K$, proving the closedness of $K$. 

\vskip .2cm  

\noindent For $n\in\mathbb{N}$ let $E_n = \llbracket w_0, w_1, \dots, w_{n+1}\rrbracket$ and set $E:=(E_n)_{n=1}^\infty$. 

\vskip .1cm 

\noindent{\bf Claim 2.} $K$ is $(r_n)$-flat with respect to $E$. Let's verify conditions (i) an (ii) of Definition \ref{dfn:1sec3}. Firstly, we prove 
\begin{itemize}
\item $K\cap E_n$ is nonempty and compact for every $n\in \mathbb{N}$. 
\end{itemize}
Fix $n\in \mathbb{N}$. To see that $K\cap E_n$ is nonempty fix any $x\in K$ and write $x=t_0w_0 + \sum_{k=1}^\infty t_k w_k$. Next define 
\begin{equation}\label{eqn:2sec4}
y=\sum_{k=0}^n t_k w_k + \Bigg( \sum_{k=n+1}^\infty t_k\Bigg)w_{n+1}.
\end{equation}
Then $y\in K\cap E_n$. Hence $K\cap E_n$ is nonempty and compact, since $K$ is bounded closed and $E_n$ is finite dimensional. 

\vskip .1cm 
\noindent We are going now to prove that:
\begin{itemize}
\item $h^E_n \leq r_{n+1}$ for every $n\in\mathbb{N}$.
\end{itemize}
 Let $n$ be fixed. Fix any $x=\sum_{k=0}^\infty t_k w_k\in K$. Then $\mathrm{dist}(x, K\cap E_n)\leq \| x - y\|$ for all $y\in K\cap E_n$. Now take $y$ as in (\ref{eqn:2sec4}) and observe that
\[
x = \sum_{k=1}^\infty \Bigg( t_0 + \sum_{i=k}^\infty t_i \Bigg) \alpha_k x_k
\]  
and
\[
y = \sum_{k=1}^{n+1}\Bigg( t_0 + \sum_{i=k}^\infty t_i\Bigg) \alpha_k x_k + t_0\sum_{k=n+2}^\infty \alpha_k x_k.
\]
From (\ref{eqn:5sec5}) we then deduce 
\[
\begin{split}
\| x - y\| &=\Bigg\| \sum_{k=n+2}^\infty \Bigg( t_0 + \sum_{i=k}^\infty t_i \Bigg) \alpha_k x_k  - t_0\sum_{i=n+2}^\infty \alpha_i x_i\Bigg\|\\[1.5mm]
&\leq 3\mu \sum_{k=n+2}^\infty \alpha_k \leq r_{n+1} 
\end{split}
\]
which shows $\mathrm{dist}(x, K\cap E_n)\leq r_{n+1}$ for all $x\in K$, so $h^E_n\leq r_{n+1}$ as desired.

\smallskip 
 
\noindent Finally, define $F\colon K\to K$  by 
\[
F\Bigg( \sum_{i=0}^\infty t_i w_i \Bigg) = \Bigg(\mu - \sum_{i=1}^\infty t_i\Bigg)w_1 + \sum_{i=1}^\infty t_i w_{i+1}.
\]
Clearly $F$ is affine and leaves $K$ invariant. Thus $\mathrm{d}(F, K)=0$. It is also easy to see that $\digamma(F)=\emptyset$. Thus $F$ is not continuous, since $K$ is a closed subset of a compact set. Now take arbitrary points $x = \sum_{k=0}^\infty t_k w_k$ and $y = \sum_{k=0}^\infty s_k w_k$ in $K$. For $i\geq 0$, set $a_i = t_i - s_i$. Then 
\[
x - y = \sum_{k=1}^\infty \Bigg( a_0 + \sum_{i=k}^\infty a_i\Bigg) \alpha_k x_k.
\]
On the other hand, direct calculation shows
\[
\begin{split}
\| F^n(x) - F^n(y)\| &\leq  \sum_{k=1}^\infty \Bigg| a_0 + \sum_{i=k}^\infty a_i\Bigg|\alpha_{k+n} + \mu \sum_{k=1}^\infty \alpha_{k+n}\\[1.5mm]
&\leq  \sup_{k\in\mathbb{N}}\Bigg| a_0 + \sum_{i=k}^\infty a_i\Bigg|\alpha_k \times \sum_{k=1}^\infty \frac{\alpha_{k+n}}{\alpha_k} + r_n\\[1.5mm]
&\leq 2\mathcal{K}\times \sum_{k=1}^\infty \frac{\alpha_{k+n}}{\alpha_k}\, \Bigg\| \sum_{k=1}^\infty \Bigg( a_0 + \sum_{i=k}^\infty a_i\Bigg) \alpha_k x_k\Bigg\| + r_n\\[1.5mm]
&= 2\mathcal{K}\times \sum_{k=1}^\infty \frac{\alpha_{k+n}}{\alpha_k}\| x - y\| +r_n.
\end{split}
\]

\noindent (ii) For $n\in\mathbb{N}$, let $r_n = 20^{\frac{n}{\alpha-1}}$. Now take the $\alpha$-H\"older retraction $R$ from $B_X$ onto $K$ given in Theorem \ref{thm:1sec3} and consider the composition $T= F\circ R$. It follows then that $\mathrm{d}(T, B_X)=0$ and
\[
\begin{split}
\| T^n(x) - T^n(y)\| &= \| F^n(R(x)) - F^n(R(y))\|\\[1.5mm]
&\leq  2\mathcal{K}\times \sum_{k=1}^\infty \frac{\alpha_{k+n}}{\alpha_k}\| R(x) - R(y)\| +r_n\leq 20^{\frac{n}{\alpha-1}}( \| x - y\|^\alpha +1),
\end{split}
\]
for all $x, y\in B_X$. This finishes the proof.
\end{proof}

\medskip 

\section{Concluding remarks and open questions}\label{sec:5}

This paper provides new fixed-point free results for Lipschitz maps. The newness is that the maps have zero minimal displacements, either on a suitable $K\in\mathcal{B}(B_X)$ or else, after applying retraction methods, on the whole ball $B_X$. In summary, our first main result (Theorem \ref{thm:M0sec4}) shows that ($\mathcal{Q}1$) can be particularly solved in every Banach space. Notice that the proof of our second result (Theorem \ref{thm:M1sec5}) is different from that given in \cite{Lin} in that we use another set $K$ and another mapping $g$. We do not know, however, whether the assumption of $(x_n)_{n=1}^\infty$ being "shift" could be removed from its statement. Here it plays a crucial role in the proof of the asymptotic regularity of the mapping $F$. Last but not least, it is important to point out that there are uniformly convex spaces with a symmetric basis which does not contain any copy of $\ell_p$ (see \cite{FJ}). Theorem \ref{thm:M3sec5} embraces the quasi-reflexive James's space $\textrm{J}_2$, solving a question posed in \cite[Open questions, p.16]{Bar}. In fact, as is well known, the summing basis of its standard unit basis forms a conditional spreading basis for $\textrm{J}_2$. With regarding to Theorem \ref{thm:M4sec5} let us point out that the mapping $F$ in statement (i) satisfies the property
\[
\limsup_{n\to\infty}\big\{ \sup\{ \|F^n(x) - F^n(y)\| - \|x- y\|\colon y\in K\}\big\}\leq 0,
\] 
for all $x\in K$. This means that $F\colon K\to K$ is asymptotically nonexpansive type (in the terminology of Kirk \cite{Ki74} see also \cite[Definition 3.5]{BR}). This should be compared with Dom\'inguez Benavides and Ram\'irez results \cite[Theorems 5.2 and 5.3]{BR}. In a certain sense, this result confers non-trivial character on what could be named as {\it fixed point theory for H\"older-contractive mappings"}. For example, it is not clear to us that the set $K$ constructed in its proof is a Lipschitz retract of $B_X$. 

\vskip .1cm 

\noindent Some further natural questions are set out in order below.

\vskip .2cm
\noindent($\mathfrak{Q}_1$) Is there $K\in\mathcal{B}(\textrm{J}_2)$ for which the FPP fails for uniformly asymptotically Lipschitz maps?

\vskip .3cm 
\noindent($\mathfrak{Q}_2$) What conditions on $X$ are sufficient to guarantee that $B_X$ fails the FPP for nonexpansive maps?

\vskip .15cm 
\noindent Let $\|\cdot\|_\gamma$ denote the Lin's $\ell_1$ renorming \cite{Lin2}. Then $B_{(\ell_1, \|\cdot\|_\gamma)}$ has the FPP for nonexpansive maps. However, notice that the unit basis of $\ell_1$ is $1$-spreading with respect to $\| \cdot\|_\gamma$. By Theorem \ref{thm:M3sec5}, for any $\varepsilon\in (0,1)$, $B_{(\ell_1, \|\cdot\|_\gamma)}$ fails the FPP for $(1+\varepsilon)$-Lipschitz maps with null minimal displacement. 

\vskip .3cm 
\noindent($\mathfrak{Q}_4$) Does $B_{(\ell_1, \|\cdot\|_\gamma)}$ fail the FPP for uniformly asymptotically regular Lipschitz maps?

\vskip .3cm 
\noindent($\mathfrak{Q}_5$) Can ($\mathcal{Q}3$) or ($\mathcal{Q}4$) be solved in any space $X$ with an unconditional basis?

\vskip .5cm  

\noindent{\bf Conflicts of Interest.} The authors declare that there are no competing interest.


  

\end{document}